\documentclass[11pt,a4paper]{article}
\date{}
\usepackage{enumerate,amsmath,amsthm,amssymb,latexsym,
amsfonts,amscd}
\usepackage[all]{xy}

\title{Differential algebras with Banach-algebra
coefficients II:\\ The operator
cross-ratio tau-function \\ and the Schwarzian derivative}

\author{
Maurice J. Dupr\'e \\ Department of Mathematics, Tulane
University\\ New Orleans, LA 70118 USA\\ mdupre@tulane.edu \\ \\
James  F. Glazebrook
\\(Primary Inst.)\\ Department of Mathematics and Computer
Science \\
 Eastern Illinois University \\
600  Lincoln Ave., Charleston, IL 61920--3099 USA \\
jfglazebrook@eiu.edu
\\ (Adjunct Faculty)
\\ Department of Mathematics \\ University of Illinois at
Urbana--Champaign\\ Urbana, IL 61801, USA
\\
\\ Emma Previato\thanks{Partial research support under grant
NSF-DMS-0808708 is very gratefully acknowledged.}
\\ Department of Mathematics and Statistics, Boston University
\\Boston, MA 02215--2411, USA \\
ep@math.bu.edu }

\theoremstyle{plain}
\newtheorem{lemma}{Lemma}[section]
\newtheorem{proposition}{Proposition}[section]
\newtheorem{theorem}{Theorem}[section]

\theoremstyle{definition}
\newtheorem{definition}{Definition}[section]

\newtheorem{remark}{Remark}[section]

\numberwithin{equation}{section}

\newcommand{\Det}{{\rm Det}}

\newcommand{\Sim}{{\rm Sim}}

\newcommand{\End}{{\rm End}}
\newcommand{\Ext}{{\rm Ext}}

\newcommand{\Gr}{{\rm Gr}}
\newcommand{\Hol}{{\rm Hol}}

\newcommand{\ID}{{\rm Id}}

\newcommand{\KP}{{\rm KP}}

\newcommand{\res}{{\rm res}}
\newcommand{\sa}{{\rm sa}}

\newcommand{\Spec}{{\rm Spec}}

\renewcommand{\a}{\alpha}
\newcommand{\be}{\beta}

\newcommand{\A}{\boldsymbol{\mathcal{A}}}

\newcommand{\cH}{\mathcal H}

\newcommand{\cL}{\mathcal L}

\newcommand{\T}{\mathcal T}

\newcommand{\bA}{\mathbb{A}}

\newcommand{\bC}{\mathbb{C}}

\newcommand{\hp}{\mathsf{H}_{+}}
\newcommand{\hm}{\mathsf{H}_{-}}
\newcommand{\kp}{\mathsf{K}_{+}}
\newcommand{\km}{\mathsf{K}_{-}}
\newcommand{\hpm}{\mathsf{H}_{\pm}}
\newcommand{\hmp}{\mathsf{H}_{\mp}}
\newcommand{\kpm}{\mathsf{K}_{\pm}}
\newcommand{\kmp}{\mathsf{K}_{\mp}}

\newcommand{\lra}{\longrightarrow}

\newcommand{\ovsetl}[1]{\overset {#1}{\lra}}

\newcommand{\wti}{\widetilde}

\newcommand{\uA}{\underline{\mathbb{A}}}

\newcommand{\cross}{\mathsf{cr}}

\newcommand{\del}{\partial}

\newcommand{\pol}{\mathfrak{P}}
\newcommand{\Schw}{\mathcal{S}}

\newcommand{\med}{\medbreak}

\begin{document}

\maketitle

\begin{abstract}
Several features of an analytic (infinite-dimensional) Grassmannian
of (commensurable) subspaces of a Hilbert space were developed in the context
of integrable PDEs (KP hierarchy). We extended some of those
features  when polarized
separable Hilbert spaces are generalized to a class of polarized
Hilbert modules, in particular the Baker and $\tau$-functions,
which become operator-valued.
Following from Part I we produce a pre-determinant
structure for a class of $\tau$-functions defined in the setting
of the similarity class of projections of a certain Banach *-algebra.
This structure is explicitly derived from the transition map
of a corresponding principal bundle.
The determinant of this map
 gives a generalized, operator-valued $\tau$-function
that takes values in a commutative C*-algebra. We extend to this setting
the operator cross-ratio which had been used to produce the
scalar-valued $\tau$-function, as well
as the associated notion of a Schwarzian
derivative along curves inside the space of similarity classes.
We link directly this cross-ratio with Fay's
trisecant identity for the $\tau$-function (equivalent to the KP
hierarchy).
By restriction to the image of the Krichever map, we use the
Schwarzian to introduce the notion of
operator-valued projective structure on a compact Riemann
surface:
 this allows a deformation inside the Grassmannian
(as it varies its complex structure). Lastly,
we use our identification of the Jacobian of the Riemann surface
in terms of extensions of the Burchnall-Chaundy C*-algebra
(Part I) to describe the KP hierarchy.

\end{abstract}

\med
\textbf{Mathematics Subject Classification (2010)}: 46L08,
53B10, 53C30, 14H70

\med
\textbf{Keywords}: Hilbert module, polarization, tau-function,
projective structure, cross-ratio, Schwarzian derivative, KP hierarchy, Fay trisecant identity.

\section{Introduction}

This paper continues from Part I \cite{DGP2} of this work where we focused on extensions
of the algebra of continuous functions $C(X)$ by compact operators,
classified via
the Calkin algebra using Brown-Douglas-Fillmore extension
theory and KK-theory to obtain results on the K-homology
of $X$ and its Jacobian.
In this Part II, we turn our attention to the connection
on the determinant bundle over the Grassmannian,
 whose curvature is related to Sato's $\tau$-function
in the scalar case, as well as in our
operator-coefficient case \cite{DGP3}, and tie these in with several concepts
developed in Part I in relationship to the KP equation. Accordingly, this Part II, as for Part I, continues
an operator-theoretic approach to the Sato-Segal-Wilson theory \cite{Sato,SW}.

We recall in \S\ref{multflow}
that the classical $\tau$-function can be realized as a
specific cocycle  of the determinant line bundle of
the universal bundle pulled back over the space of restricted
polarizations from the restricted Grassmannian and its dual
\cite{Sato,SW}. The specific cocycle is determined by two
non-vanishing sections. The
underlying cocycle for the universal bundle is used to arrive at
the \emph{pre-determinant} structure of the $\tau$-function which we
refer to as the $\mathfrak T$-\emph{function}; accordingly, its determinant
$\Det(\mathfrak T)$ is the $\tau$-function (see
\cite{DGP3,MSW, Zelikin}). The important structure lies
within the geometry of this determinant bundle, its connection and
its curvature.  But the $\mathfrak T$-function is the cocycle for the
universal bundle over a space of restricted polarizations $\pol$,
relating essentially the same two underlying sections, in this case, of
the associated principal bundle.
Hence the interest is in the
calculation of the geometry, connection, and curvature of the
principal bundle of the universal bundle using the two sections
which are each covariantly constant over two complementary
subbundles of the tangent bundle of the space of restricted
polarizations \cite{DGP3,Zelikin}.

In \cite{DGP1,DGP2} we considered a certain (complex) Banach *-algebra
$A$ modeled on the linear operators of a Hilbert module denoted
$H_{\A}$, where $\A$ is a commutative separable C*-algebra.
Letting $P(A)$ denote the idempotents in $A$, in \cite{DGP3}
we considered the geometry of the space
$\Lambda = \Sim(p,A)$, the similarity class of $p\in P(A)$,
which is closely related
to the Grassmanian $\Gr(p,A)$ of Part I \cite{DGP2}
(see also \cite{DGP1,DGP3}).
 From the transition map
of a principal bundle $V_{\Lambda} \lra \Lambda$, we
 deduced a corresponding pre-determinant
denoted $\T$. We identify in \S\ref{multflow} a relationship between $\T$ and
 $\mathfrak T$
obtained via a diffeomorphism between $\Lambda$ and $\pol$.
In fact, as shown in \cite{DGP3}, certain calculations involving the
connection and curvature
are more straightforward when performed on
$\Lambda$, and that is why we choose to work mainly with predeterminants
 $\T_{\lambda}$ and the
ensuing determinants $\tau_{\lambda}$.

We revisit the construction of the $\tau$-function in terms of a cross-ratio
(\S\ref{opcrossratio}),  and of the Schwarzian derivative that
preserves it (\S\ref{cross-ratio}).
 In doing so, we need to extend to the C*-algebra case
an (abelian) group action on the Grassmannian (\S\ref{group}),
study the pull-back of the universal bundle through it,
compare this to the Poincar\'e bundle \cite{Alvarez}
(\S\ref{poincare}), and relate the operator cross-ratio
 to the projective structure on the Riemann surface associated
to the Schwarzian derivative (\S\ref{proj-struct}).
We then switch (\S\ref{wick})
to Raina's interpretation of the KP hierarchy
(Fay's trisecant identity) in model quantum-field theory \cite{Raina}:
 he reduced the identity to Wick's theorem,
by writing the 4-point function in terms of theta functions and prime forms,
essentially  a (generalized) cross-ratio
on a Riemann surface.
As a result, we  are able to give in \S\ref{k-dv} a KP hierarchy
satisfied on the space of extensions of the Burchnall-Chaundy C* algebra.

We retain the notation of Part I and refer there to the relevant concepts and results obtained.
The Appendix to Part I recalls much of the background material which is used here. The new results obtained
in this Part II are Lemma \ref{gamma-action}, Propositions \ref{T-diagram}, \ref{Det-diagram}, \ref{2-forms}, \ref{lambda-analytic},
\ref{proj-prop}, Theorems \ref{main-2}, \ref{main-1}.


\section{Universal bundles with connection and the $\T$-function}\label{previous}

\subsection{The space of polarizations}\label{polarizations}

Let $A$ be a unital complex Banach(able) algebra with group of units $G(A)$ and space
of idempotents $P(A)$.
Let $H$ be a separable (infinite dimensional) Hilbert space (here as
in Part I we take
$H = L^2(S^{1}, \bC)$). Given a unital separable C*-algebra $\A$, we take the
standard (free countable dimensional) Hilbert module $H_{\A}$ over
$\A$ and consider a \emph{polarization} of $H_{\A}$ given by a pair of
 submodules
$(\hp, \hm)$, such that
\begin{equation}\label{polar1}
H_{\A} = \hp \oplus \hm ~,~\text{and}~ \hp \cap \hm = \{0\}.
\end{equation}
Recall also from \cite{DGP1,DGP2} the (restricted) Banach *-algebra $A =
\cL_J(H_{\A})$
($J$ being a unitary $\A$--module map  satisfying
$J^2 = 1$)
which henceforth we use.
As in \cite{DGP2}, we assume $\A$ to be commutative (and separable). The Gelfand transform
implies there exists a compact
metric space $Y$ such that $Y = \Spec(\A)$ and $\A \cong C(Y)$, and thus
we note that
\begin{equation}\label{aform}
A \cong \{\text{continuous functions}~ Y \lra B
\} = C(Y,B),
\end{equation}
where $B = \cL_J(H)$ corresponds to taking $\A= \bC$.

We recall the Grassmannian $\Gr(p,A)$ and refer to \cite[A.2]{DGP2} for the necessary background.
For a given $p \in P(A)$, there is associated to $\Gr(p,A)$ its
\emph{dual Grassmannian} $\Gr^{*}(p,A)$ \cite[\S3]{DGP3}. Let $\pol$
denote the space of polarizations $(\hp, \hm)$ on $H_{\A}$. Then as shown in
\cite[\S3]{DGP3}, the space $\pol$ can be regarded as a subspace
\begin{equation}
\pol \subset \Gr(p, A) \times \Gr^*(p, A).
\end{equation}
A significant observation is that $\pol$ can be
closely related to the similarity class $\Lambda = \Sim(p,A)$ of $A$
where $\Lambda$ consists of the elements of $\pol$ expressed in terms of
projections.
In fact, $\Lambda$ admits a natural complex analytic structure induced
from that of $A$ via $P(A)$ (see \cite{DG2,DGP3} and Proposition \ref{lambda-analytic} below).
Further, from \cite[Theorem 4.1(3)]{DGP3} there exists
an analytic diffeomorphism
\begin{equation}\label{diffeo-1}
{\phi} : \pol \lra \Lambda \subset P(A).
\end{equation}
\begin{remark}\label{sa-remark-1}
Given that $\Lambda$, as a set of idempotents in the Banach *-algebra $A$, can admit non self-adjoint
elements (see \cite[Remark 2.1]{DGP2}), we note that
it will not be diffeomorphic to $\Gr(p,A)$
under the natural quotient map of equivalence relations $\Pi : P(A) \lra \Gr(A)= P(A)/{\sim}$ . However, restriction of $\Pi$  to the
self-adjoint elements $\Lambda^{\sa}$ of $\Lambda$ can be shown to
establish a bijective
diffeomorphism between $\Lambda^{\sa}$ and $\Gr(p,A)$, and moreover,
there exists a
smooth retraction of $P(A)$ onto $P^{\sa}(A)$ (as can be shown using the technique of \cite[Proof of Proposition 4.6.2]{Bl}).
\end{remark}

\subsection{The universal bundle with connection over $\Lambda$}

Here we briefly describe part of the basic geometry of \cite{DGP3}.
Firstly, let $\pi_{\Lambda} =\Pi \vert \Lambda$ and
$\pi_V=\Pi\vert V(p,A)$.
Let $V_{\Lambda}=\pi_{\Lambda}^*(V(p,A))$; specifically,
\begin{equation}
V_{\Lambda}= \{(r,u) \in {\Lambda} \times V(p,A):
\pi_{\Lambda}(r)=\pi_V(u)\}.
\end{equation}
In \cite{DGP3} we constructed an analytic
principal right $G(pAp)$-bundle with connection
\begin{equation}\label{lambda-principal}
(V_{\Lambda}, \omega_{\Lambda}) \lra \Lambda,
\end{equation}
whose (analytic) $G(pAp)$-valued transition map $t_{\Lambda}$ is given by
the formula
\begin{equation}\label{trans2}
t_{\Lambda}((r,u),(r,v))=t_V(u,v),
\end{equation}
where $t_V$ is the transition map for the $G(pAp)$-bundle $V(p,A) \lra \Gr(p,A)$ (see \cite[A.2]{DGP2}).

By the standard means we have the associated vector bundle with (Koszul)
connection, namely
\emph{the universal bundle with connection over $\Lambda$}
\begin{equation}\label{lambda-universal}
(\gamma_{\Lambda}, \nabla_{\Lambda}) \lra \Lambda,
\end{equation}
for which the curvature operator $R_{\nabla}$ was computed explicitly in \cite[\S 8]{DGP3}.

\subsection{The $\T$-function}\label{opcrossratio}

For given parallel (covariantly constant) sections $\a_p, \be_p$ of \eqref{lambda-principal}, we
defined the \emph{$\T$-function} in terms of the transition map $t_{\Lambda}$ of
\eqref{trans2} by
\begin{equation}\label{T-FUNCTION-1}
\T(r)= t_{\Lambda}(\alpha_p(r),\beta_p(r)),
\end{equation}
which can be expressed
 more conveniently as $\T= t_{\Lambda}(\alpha_p,\beta_p)$ (see \cite[(8.1)]{DGP3}).

Returning to the space of polarizations $\pol$, there are several closely
 associated objects
as described in \cite{DGP3,Zelikin}. Firstly, there is the principal
bundle with connection
\begin{equation}\label{pol-principal}
(V_{\pol}, \omega_{\pol})  \lra \pol,
\end{equation}
and associated universal (vector) bundle with connection
\begin{equation}\label{pol-universal}
(\gamma_{\pol}, \nabla_{\pol})  \lra \pol.
\end{equation}
Observe that there is a splitting
\begin{equation}\label{pmsplit}
d = \del_{+} + \del_{-},
\end{equation}
of the exterior derivative of $\omega_{\Lambda}$ induced in the
following way.
For a given polarization $\mathcal P =
(\hp, \hm) \in \pol$, the exterior derivative of $\omega_{\pol}$
splits as
$d = \del_{+} + \del_{-}$ where $\del_{+}$ (respectively, $\del_{-}$)
denotes the covariant
derivative in the directions from $\hp$ (respectively, $\hm$).
The induced splitting \eqref{pmsplit}
thus follows from the analytic diffeomorphism in \eqref{diffeo-1}
for which $\omega_{\Lambda} = ({\phi}^{-1})^* \omega_{\pol}$.

In \cite{Zelikin}, the function $\mathfrak T$ constructed via
operator cross-ratio, is likewise shown to be derived from the transition
map $t_{\pol}$ for
$V_{\pol} \lra \pol$. This is closely related to our $\T$-function, and in
 \cite{DGP3}, we showed that the
geometry of $(V_{\Lambda} , \omega_{\Lambda}) \lra \Lambda$ is also closely
related to that of  $(V_{\pol} , \omega_{\pol}) \lra \pol$,
where local coordinates for $\Lambda$ and $\pol$ can be expressed
in terms of an \emph{operator cross-ratio} (see \S\ref{cross-ratio} below).
 In particular,
with reference to \cite[p. 47]{Zelikin}
and recalling the analytic diffeomorphism
${\phi} : \pol \lra \Lambda$ in
\eqref{diffeo-1},
we have in terms of the parallel sections
$\a_p$ and $\be_p$, the relationship
\begin{equation}\label{T-relation-1}
\phi^* \T = \phi^*(t_{\Lambda}(\alpha_p,\beta_p)) = \be_p^{-1}
\a_p = \mathfrak{T}.
\end{equation}

Consider a pair of polarizations $(\hp, \hm), (\kp,
\km) \in \pol$. Let $\hpm$ and $\kpm$ be `coordinatized' via maps
$P_{\pm} : \hpm \lra \hmp$, and $Q_{\mp}: \kpm \lra \kmp$. The composite map
\begin{equation}\label{composite}
\hp \ovsetl{\km} \kp \ovsetl{\hm} \hp,
\end{equation}
allows us to take
 the operator cross--ratio \cite{Zelikin} (cf. \cite{DGP3}):
\begin{equation}\label{tauop1}
\mathfrak{T} (\hp, \hm; \kp, \km) = (P_{-}P_{+} -
1)^{-1}(P_{-}Q_{+} - 1) (Q_{-}Q_{+} - 1)^{-1} (Q_{-}P_{+} - 1).
\end{equation}
 For this construction there is no essential algebraic change in
generalizing from polarized Hilbert spaces to polarized Hilbert
modules. The principle here is that the transitions between charts
define endomorphisms of $W \in \Gr(p,A)$ that will become the
transition functions of the universal bundle $\gamma_{\pol} \lra
\pol$.
The main properties of $\T$ and $\mathfrak T$ is that they are
pre-determinants
for various classes of $\tau$-functions, as we shall see in \S\ref{multflow}.

\subsection{Trace-class operators and the determinant}\label{trace-class}

An equivalent, operator, description leading to
the functions $\T$ and $\mathfrak{T}$ above, can be obtained
along the lines of e.g. \cite{MSW,SW,Zelikin}. Here, as in \cite{DGP2},
we make use of the nested sequence
of Schatten ideals definable in $\cL(H_{\A})$ (see
e.g. \cite{Smith}). Suppose $(\hp, \hm), (\kp, \km) \in \pol$ are such
that $\hp$ is the graph of a linear map $S : \kp \lra \km$ and
$\hm$ is the graph of a linear map $T: \km \lra \kp$~. Then on
$H_{\A}$ we consider the identity map $\hp \oplus \hm \lra \kp
\oplus \km$, as represented in the block form
\begin{equation}\label{tauop2}
\bmatrix a &b \\ c &d
\endbmatrix
\end{equation}
where $a: \hp \lra \kp,~ d: \hp \lra \km$ are zero--index Fredholm
operators, and $b: \hp \lra \kp,\\ ~ c: \hp \lra \km$ are in
$\mathcal K(H_{\A})$ (the compact operators), such that $S =
ca^{-1}$ and $T = bd^{-1}$. Initially, one considers the operator
$1 - ST = 1 - ca^{-1}bd^{-1}$. In particular, with a view to
defining a generalized determinant leading to \emph{an
operator--valued $\tau$-function}, we consider cases where
ST is of \emph{trace class}. If we take $b,c$ to be
Hilbert--Schmidt operators (as for the case $\A = \bC$ as in
\cite{MSW,SW,Zelikin}), then $ST$ is of trace--class, the
operator $(1 - ST)$ is $\mathfrak{T}(\hp, \hm; \kp,
\km)$ above, and the $\tau$-function is defined as
\begin{equation}\label{tauop3}
\Det ~\mathfrak{T}(\hp, \hm; \kp, \km) \otimes 1_{\A}
= \Det (1 - ca^{-1}bd^{-1}) \otimes 1_{\A} = \Det (\alpha_p \beta^{-1}_p).
\end{equation}
Starting from the universal bundle $\gamma_{\Gr} \lra \Gr(p,A)$,
then with respect to an `admissible basis' (for defining determinants) in the Stiefel bundle $V(p,A) \lra \Gr(p,A)$
(see \cite{DG2} and \cite[A.2]{DGP2}), the
$\tau$-function in \eqref{tauop3} is equivalently derived from the
canonical section of the determinant line bundle $\Det(\gamma^*_{\Gr}) \lra \Gr(p,A)$
(cf. \cite{MSW,SW,Zelikin}).


\section{Relationship between the predeterminants and
$\tau$-functions}\label{multflow}

\subsection{Relationship between the predeterminants}\label{group}

We start with the following lemma:
\begin{lemma}\label{gamma-action}
Let $\Gamma \subset G(pAp)$ be a group acting on the subspace $\hp$. Then
there exists a natural action of $\Gamma$ on the spaces $\Gr(p,A)$,
$\Lambda= \Sim(p,A)$ and $\pol$.
\end{lemma}

\begin{proof}
Firstly, $\hp$
is a splitting subspace for $H_{\A}$ and thus determines
the polarization $H_{\A} = \hp \oplus \hm$, and likewise
for any other polarizing pair $(\kp, \km)$.

We recall from \cite[\S4]{DG2} that $\Gr$ is a functor, and if
$g$ is a linear automorphism of $\hpm$, it thus defines an element of
$G(pAp)$. For $p \in P(A)$, let $\hat{p} = 1 -p$.
Then $g + \hat{p} \in G(A)$, and therefore this term defines an inner
automorphism of $A$ taking $\Lambda$ to
itself, fixing $p$ and inducing an analytic diffeomorphism of $\Gr(p,A)$
with itself. Thus by functorial properties
of inner automorphisms and the functorial properties of $\Gr$, the group
$\Gamma$ acting on $\hp$, viewed as a subgroup of
$G(pAp)$, induces actions on $\Lambda$ and $\Gr(p,A)$.

Further, given an action of $\Gamma$ on $\Lambda$, there is an induced action
on the space of polarizations $\pol \subset
\Gr(p,A) \times \Gr^*(p,A)$, since $\Lambda$ simply consists of
the elements of $\pol$ viewed as projections.
\end{proof}
The induced action of $\Gamma$ gives
rise to following commutative diagram
\begin{equation}\label{Gamma-pic}
\begin{CD}
\Gamma \times \pol  @> \mu_{\pol} >> \pol
\\ @V \bar{\phi} VV   @VV {\phi}  V \\
\Gamma \times \Lambda     @> \mu_{\Lambda} >> \Lambda
\\ @V \bar{\pi}_{\Lambda} VV   @VV \pi_{\Lambda} V \\
\Gamma \times \Gr(p,A)     @> \mu_{\Gr} >> \Gr(p,A)
\end{CD}
\end{equation}
where we note that the action of $\Gamma$ as a subgroup of $G(pAp)$ on
$\Gr(p,A)$ is essentially trivial since $G(pAp)$ is
the structure group of the principal bundle $V(p,A) \lra \Gr(p,A)$.
An example of such a group $\Gamma$ is given by the group of multiplication
operators $\Gamma_{+}(\A)$ in \cite{DGP2} (in particular $\Gamma_{+}$ in
\cite{SW}
 for $\A = \bC$).
 In the following we will make use of the various maps appearing in
\eqref{Gamma-pic} besides recalling the role of canonical sections
for the universal bundles $\gamma_{\Gr} \lra \Gr(p,A)$ (the associated vector bundle
to $V(p,A) \lra \Gr(p, A)$, see \cite[A.2]{DGP2}) and
$\gamma_{\Lambda} \lra \Lambda$, introduced in \cite[\S7]{DGP3}.
Let $S_{p}$ be the canonical section of $\gamma_{\Gr}$ and set $\wti{S}_{p}
= \mu_{\Gr}^*(S_{p})$. Likewise, let $S_{\lambda}'$ be the canonical
section of $\gamma_{\Lambda}$ and set  $\wti{S}'_{\lambda}
= \mu_{\Lambda}^*(S_{\lambda}')$. The following implements the sections
functor $H^0$.

\begin{proposition}\label{T-diagram}
Let $W \in \Gr(p,A)$ and $\lambda \in \Lambda$. In terms of the maps
in \eqref{Gamma-pic} the following diagram is
commutative where the horizontal maps
$\rho_1, \rho_2$ are homomorphisms
and the vertical maps $\bar{\pi}^*_{\Lambda},
 (\bar{\pi}^*_{\Lambda})_{\res}$ are isomorphisms:
\begin{equation}
\begin{CD}
H^{0}(\Gamma \times \Gr(p,A), ~\mu^*_{\Gr} \gamma_{\Gr}) @> \rho_1
>> H^{0}(\Gamma \times \{W \}, ~\mu^*_{\Gr} \gamma_{\Gr}\vert \Gamma
\times \{W \})
\\ @V \bar{\pi}^*_{\Lambda} V \cong V   @VV
(\bar{\pi}^*_{\Lambda})_{\res} \cong V\\
H^{0}(\Gamma \times \Lambda, ~\mu^*_{\Lambda} \gamma_{\Lambda}) @>
\rho_2
>> H^{0}(\Gamma \times \{\lambda\}, ~\mu^*_{\Lambda} \gamma_{\Lambda}
\vert \Gamma \times \{\lambda\})
\end{CD}
\end{equation}
In particular, we have $ \bar{\pi}^*_{\Lambda}(\wti{S}_{\lambda}) =
\wti{S}'_{p}$.
\end{proposition}

\begin{proof}
The commutativity follows by applying the sections functor $H^{0}$
to the lower square in \eqref{Gamma-pic}. Since
 we have $\gamma_{\Lambda} \cong \pi^*_{\Lambda}
\gamma_{\Gr}$, it follows that
\begin{equation}
\begin{aligned}
(\mu_{\Gr} \circ \bar{\pi}_{\Lambda})^* \gamma_{\Gr} &\cong
(\pi_{\Lambda} \circ \mu)^* \gamma_{\Gr} \\
&\cong \mu_{\Lambda}^*( \pi^*_{\Lambda} \gamma_{\Gr})
\cong \mu_{\Lambda}^*\gamma_{\Lambda}.
\end{aligned}
\end{equation}
Thus $\bar{\pi}_{\Lambda}^*(\mu_{\Gr}^* \gamma_{\Gr}) \cong
\mu_{\Lambda}^* \gamma_{\Lambda}$, and likewise for the restriction
to $\Gamma \times \{W\}$.
\end{proof}

With regards to the maps in Proposition \ref{T-diagram}, we next introduce
the following:
\begin{definition}\label{def-T-1}
For $W \in \Gr(p,A)$, we denote by $\mathfrak{T}_W$ \emph{the
 $\mathfrak{T}$-function of the point $W$ over $\Gamma$}, as defined
to be the image of the section $\wti{S}_{p} = \mu^*_{\Gr}(S_{p})$ under $\rho_1$.
\end{definition}

\begin{definition}\label{def-T-2}
For $\lambda \in \Lambda$, we denote by $\T_{\lambda}$ \emph{the
 $\T$-function of the point $\lambda$ over $\Lambda$}, as defined
to be the image of the
section $\wti{S}'_{\lambda} = \mu^*_{\Lambda}(S'_{\lambda})$ under $\rho_2$.
\end{definition}
Then it is easily seen from Proposition \ref{T-diagram} that we have
\begin{equation}\label{T-relation}
(\bar{\pi}^*_{\Lambda})_{\res}(\mathfrak{T}_W) = \T_{\lambda}.
\end{equation}
Recalling the analytic diffeomorphism ${\phi}: \pol \lra
\Lambda$ in \eqref{diffeo-1}, the relationship \eqref{T-relation} is
another way of interpreting our observation that ${\phi}^* \T
= \mathfrak{T}$.

In terms of the group multiplication $m: \Gamma \times \Gamma \lra \Gamma$
in $\Gamma$, we have the following
commutative diagram
\begin{equation}\label{Gamma-times}
\begin{CD}
\Gamma \times
\Gamma \times \Lambda @> {m \times \ID} >> \Gamma \times \Lambda
\\ @VVV  @VVV \\
\Gamma \times
\Gamma \times \Gr(p,A) @> {m \times \ID} >> \Gamma \times \Gr(p,A)
\end{CD}
\end{equation}

\begin{itemize}
\item[(1)]
For $W \in \Gr(p,A)$, let $\mu_W : \Gamma \times \{W\} \lra
\Gr(p,A)$ denote the map induced by $\mu_{\Gr}$ in
\eqref{Gamma-pic}.

\item[(2)]
For $\lambda \in \Lambda$, let $\mu_{\lambda} : \Gamma \times
\{\lambda\} \lra \Lambda$ denote the map induced by $\mu_{\Lambda}$
in \eqref{Gamma-pic}.
\end{itemize}
Thus for each point $W \in \Gr(p,A)$ (respectively, $\lambda \in
\Lambda$) we obtain vector bundles over $\Gamma \times \Gamma$
associated with $W$ (respectively, $\lambda$), as given by
\begin{equation}\label{gamma-cross1}
\begin{aligned}
\mathfrak{E}_W &= (1 \times \mu_{W})^*
(\mu_{\Gr}^* \gamma_{\Gr}) =
m^*(\mu_{\Gr}^* \gamma_{\Gr}), \\
\mathfrak{E}_{\lambda}&= (1 \times \mu_{\lambda})^*(\mu_{\Lambda}^*
\gamma_{\Lambda})
= m^*(\mu_{\Lambda}^* \gamma_{\Lambda}).
\end{aligned}
\end{equation}

\subsection{The Poincar\'{e} bundles and the
$\tau$-function}\label{poincare}

Returning to the $\tau$-functions considered in Part I
\cite{DGP2}, we next make
some corresponding observations for determinants using admissible
bases in $V(p,A)$
(cf. \cite{Alvarez}).
Here, for ease of notation, we  set $\Gamma_1 = \Gamma_{+}(\A)$,
the action we are concerned with (cf Lemma \ref{gamma-action}).
In this case it will be convenient to use the
following notation:
\begin{equation}\label{Dets}
\Det_{\Gr}= \Det(\gamma^*_{\Gr})~ \text{and} ~ \Det_{\Lambda}=
\Det(\gamma^*_{\Lambda}),
\end{equation}
noting that $\Det_{\Lambda} = \pi_{\Lambda}^* \Det_{\Gr}$. These we
 will pull-back by maps
\begin{equation}
\begin{aligned}
\hat{\mu}_{\Gr}&: J_{\A}(X) \times \Gr(p,A) \lra \Gr(p,A), \\
\hat{\mu}_{\Lambda}&: J_{\A}(X) \times \Lambda \lra \Lambda,
\end{aligned}
\end{equation}
where we recall the space $J_{\A}(X)$ of monomorphisms $\bA \otimes
\A \lra B_W$ with respect to the spectral
curve $X$
(see \cite[Appendix A.4]{DGP2}). In keeping with some standard
algebraic-geometric terminology (cf. \cite{Alvarez}), the line bundles
\begin{equation}\label{poincare-0}
\begin{aligned}
\hat{\mu}^*_{\Gr} \Det_{\Gr} &\lra J_{\A}(X) \times \Gr(p,A), \\
\hat{\mu}^*_{\Lambda} \Det_{\Lambda} & \lra J_{\A}(X) \times \Lambda,
\end{aligned}
\end{equation}
are  referred to as \emph{Poincar\'{e} bundles}.
Next we pull-back the maps in \eqref{poincare-0} along the map $\Gamma_1
\lra J_{\A}(X)$ in \cite[A.19]{DGP2} to
obtain
\begin{equation}\label{poincare-1}
\begin{aligned}
\mu^*_{\Gr} \Det_{\Gr} &\lra \Gamma_1 \times \Gr(p,A), \\
\mu^*_{\Lambda} \Det_{\Lambda} &\lra \Gamma_1 \times \Lambda.
\end{aligned}
\end{equation}
\begin{remark}
By incorporating the group multiplication, we have in a similar
way to \eqref{gamma-cross1} the following
Poincar\'{e} bundles over
${\Gamma}_1 \times {\Gamma}_1$ associated with $W \in \Gr(p,A)$
(respectively, $\lambda \in
{\Lambda}$)
\begin{equation}\label{poincare-2}
\begin{aligned}
\mathfrak{B}_W &= (1 \times \mu_{W})^*(\mu_{\Gr}^* \Det_{\Gr}) =
m^*(\mu_{\Gr}^* \Det_{\Gr}), \\
\mathfrak{B}_{\lambda}&= (1 \times \mu_{\lambda})^*(\mu_{\Lambda}^*
\Det_{\Lambda}) = m^*(\mu_{\Lambda}^* \Det_{\Lambda}).
\end{aligned}
\end{equation}
\end{remark}
The next step is to give the analogous statement to Proposition
\ref{T-diagram}
for determinants and thus relate the $\tau$-functions corresponding
to respective points
of $\Gr(p,A)$ and $\Lambda$. Specifically, for fixed $W \in \Gr(p,A)$,
 let us set
$\wti{\cL}_{\tau}(W) = \mu^*_{\Gr} \Det_{\Gr}\vert \Gamma_1 \times \{W
\}$, and likewise, for fixed $\lambda \in \Lambda$, let us set
$\wti{\cL}_{\tau}(\lambda) = \mu^*_{\Lambda} \Det_{\Lambda}\vert
\Gamma_1 \times \{\lambda \}$.

Let $Q_{p}$ be the canonical section of $\Det_{\Gr}$ and set
$\wti{Q}_{p} = \mu_{\Gr}^*(Q_{p})$. Likewise, we take $Q_{\lambda}'$ to be
the canonical section of $\Det_{\Lambda}$ and set $\wti{Q}'_{\lambda} =
\mu_{\Lambda}^*(Q_{\lambda}')$.
Motivated by \cite[\S5]{Alvarez} and \cite{SW}, we introduce a
$\tau$-function $\tilde{\tau}_W$ associated to $W \in \Gr(p,A)$
defined by $\tilde{\tau}_W = \rho_1(\wti{Q}_p)$, and associated
to $\lambda \in \Lambda$, we likewise
define $\tilde{\tau}_{\lambda}= \rho_2(\wti{Q}_{\lambda}')$.

\begin{proposition}\label{Det-diagram}
Let $W \in \Gr(p,A)$ and $\lambda \in \Lambda$. With regards to the
maps in \eqref{Gamma-pic} we have the following commutative diagram
in which the horizontal maps $\rho_1, \rho_2$ are homomorphisms
and the vertical maps
$\bar{\pi}^*_{\Lambda},(\bar{\pi}^*_{\Lambda})_{\res} $
are isomorphisms:
\begin{equation}
\begin{CD}
H^{0}(\Gamma_1 \times \Gr(p,A), ~\mu^*_{\Gr} \Det_{\Gr}) @> \rho_1
>> H^{0}(\Gamma_1 \times \{W \}, ~\wti{\cL}_{\tau}(W))
\\ @V \bar{\pi}^*_{\Lambda} V \cong V   @VV (\bar{\pi}^*_{\Lambda})_{\res}V\\
H^{0}(\Gamma_1 \times \Lambda, ~\mu^*_{\Lambda} \Det_{\Lambda}) @>
\rho_2
>> H^{0}(\Gamma_1 \times \{\lambda\}, ~ \wti{\cL}_{\tau}(\lambda))
\end{CD}
\end{equation}
In particular, we have $\tilde{\tau}_W = \Det~ \mathfrak{T}_W =
\rho_1(\wti{Q}_p)$ and $\tilde{\tau}_{\lambda} = \Det~ \T_{\lambda}
= \rho_2(\wti{Q}_{\lambda}')$.
\end{proposition}
\begin{proof}
This follows immediately from Proposition \ref{T-diagram} when the
coefficients are in the respective determinant bundle, and from the
definitions of the $\mathfrak{T}_W$ and $\T_{\lambda}$-functions
once their respective determinants are taken.
\end{proof}

We also recall the notion of `transverse subspace' from \cite{DGP2,SW}.
We note that one way of characterizing $W$'s transverse
to $H_-$ is that the orthogonal projection
$W\longrightarrow H_+$ should be an isomorphism; this
property is not preserved in general under multiplication by
an element of the group $\Gamma_1= \Gamma_+(\A)$,
but it is preserved over a dense subset
which we denote by $\Gamma_1^W = \Gamma_+^W(\A)$
(see \cite[Appendix A.5]{DGP2}).
Likewise, we set $\Gamma_1^{\lambda} =\Gamma_+^{\lambda}(\A)$.

In a similar way to \cite{SW} (cf. \cite{Alvarez}) we define
operator-valued $\tau$-functions
relative to such subspaces
$W \in \Gr(p,A)$ and $\lambda \in \Lambda$ as follows:
\begin{itemize}
\item[(1)]
Fix a transverse subspace $W \in \Gr(p,A)$ and define
\begin{equation}
{\cL}_{\tau}(W) = \wti{\cL}_{\tau}(W) \vert {\Gamma}_1^W \times \{W \}.
\end{equation}
Let $\sigma_W$ be a constant section
 trivializing ${\cL}_{\tau}(W)$
over ${\Gamma_1^W}$ (it is to ensure the existence
of this section that we take $W$ transverse \cite[Prop. 3.3]{SW}).
Recalling from
Proposition \ref{Det-diagram} that
 we have $\tilde{\tau}_W = \Det~ \mathfrak{T}_W =
\rho_1(\wti{Q}_{p})$, and for $g \in {\Gamma_1^W}$, we define
$\tau_W: \Gamma_1^W
\lra \bC \otimes 1_{\A}$ by
\begin{equation}
\tau_W(g) = \tilde{\tau}_{W}(g) (\sigma_W(g))^{-1},
\end{equation}
which simply recovers the tau-function $\tau_W$ of \cite{DGP2}
(cf. \cite{SW}) with $\A$-valued coefficients.

\item[(2)]
Fix a transverse subspace $\lambda \in \Lambda$ and define
\begin{equation}
{\cL}_{\tau}(\lambda) = \wti{\cL}_{\tau}(\lambda)
\vert {\Gamma}_1^{\lambda} \times \{ \lambda \}.
\end{equation}
Let $\sigma_{\lambda}$ be a constant section
trivializing ${\cL}_{\tau}(\lambda)$
over ${\Gamma_1^\lambda}$. Recalling from Proposition \ref{Det-diagram} that
 we have $\tilde{\tau}_{\lambda}= \Det~ \mathcal{T}_{\lambda} =
\rho_2(\wti{Q}_{\lambda}')$, and for $g \in {\Gamma_1^\lambda}$,
we define $\tau_\lambda:
\Gamma_1 \lra \bC \otimes 1_{\A}$ by
\begin{equation}\label{tau-lambda-1}
\tau_{\lambda}(g) = \tilde{\tau}_{\lambda}(g) (\sigma_{\lambda}(g))^{-1}.
\end{equation}
\end{itemize}
This leads to the straightforward relationship
\begin{equation}\label{tau-lambda}
(\bar{\pi}^*_{\Lambda})_{\res}~ (\tau_W(g)) = \tau_{\lambda}(g).
\end{equation}

At this stage it should be clear from the above results
 that the geometry of
$\Lambda$ and $\Gr(p,A)$, as well as the functions $\tau_{\lambda},
\tau_W$, are  closely related via $\pi_{\Lambda}$ and \eqref{diffeo-1}.
In particular,  by following standard procedures, $(\gamma_{\Lambda},
\nabla_{\Lambda})
\lra \Lambda$ induces the determinant line bundle with its connection
$(\Det_{\Lambda},
\nabla(\Det_{\Lambda}) )\lra \Lambda$, and likewise for
$(\Det_{\Gr}, \nabla(\Det_{\Gr})) \lra \Gr(p,A)$. Here the $\tau$-function
serves as a
`logarithmic potential' for the curvature of the connection, and from
\cite{MSW,Zelikin} (cf. \cite{DGP3}), we have recalling \eqref{pmsplit} for
the curvature 2-forms:
\begin{equation}\label{curv-2-form}
\Omega(\Det_{\Lambda}) = \frac{1}{2\pi \iota} \del_{+}
\del_{-} \log \vert \tau_{\lambda} \vert, ~ \text{and}~
\Omega(\Det_{\Gr}) = \frac{1}{2\pi \iota} \del_{+} \del_{-}
\log \vert \tau_{W} \vert.
\end{equation}
Since the corresponding calculations of the connection and curvature
of $(\gamma_{\Gr}, \nabla_{\Gr}) \lra \Gr(p,A)$ are more straightforward
on passing to
$(\gamma_{\Lambda}, \nabla_{\Lambda}) \lra \Lambda$ (see \cite{DGP3}
for details), we choose to emphasize objects relative to $\Lambda$, such
as $\T_{\lambda}, \tau_{\lambda}$, etc. in the following.

Specifically, we can use the canonical section $S_{\lambda}'$
to lift the action of
$\Gamma_1$ to
the universal bundle $\gamma_{\Lambda}$. Moreover, in order to
give a more explicit expression for
$\T_{\lambda}$, we use parallel sections
$\a_{\lambda}, \be_{\lambda}$ as in
\eqref{T-FUNCTION-1}, and then  $\T_{\lambda}: \Gamma_1 \lra
G(\lambda A \lambda)$ is equivalently
defined by
\begin{equation}\label{T-lambda-1}
\T_\lambda(g)(r) = t_{\Lambda}(g^{-1} \a_{\lambda}(r), \be_{\lambda}(rg)),
\end{equation}
for $g \in \Gamma$. In the following we shall simply drop
the argument in $r$ since this is understood.

On recalling the element $q_{\zeta} \in {\Gamma}_1$ as given by a
map $q_{\zeta}(z)= (1- z\zeta^{-1})$ in \cite[\S4.4]{DGP2},
we proceed, in view of the `transversality' to
define for $g\in \Gamma_1^\lambda$,
\begin{equation}\label{T-lambda-2}
\Psi_{\lambda}(g, \zeta) = \T_{\lambda}(g \cdot q_{\zeta})
(\T_{\lambda}(g))^{-1}.
\end{equation}
Observe that the function $\Psi_{\lambda}$ is a `predeterminant' for a
Baker function $\psi_{\lambda}$
in the following sense. On taking determinants of \eqref{T-lambda-2}, we
obtain
\begin{equation}\label{T-lambda-3}
\psi_{\lambda}(g, \zeta): = \Det (\Psi_{\lambda}(g, \zeta)) =
\tau_{\lambda}(g \cdot q_{\zeta}) (\tau_{\lambda}(g))^{-1},
\end{equation}
which is simply a `lifted-to-$\Lambda$' version of the relationship
between the Baker $\psi_W$ and $\tau_{W}$ functions
under pull-back by $\pi_{\Lambda}$ (see \cite[\S4.4]{DGP2} and \cite{SW}):
\begin{equation}
\psi_{W}(g, \zeta) = \tau_{W}(g \cdot q_{\zeta}) (\tau_{W}(g))^{-1}.
\end{equation}



\section{Applications}


\subsection{Operator cross-ratio and the Schwarzian
derivative}\label{cross-ratio}

Smooth and analytic parametrizations of subspaces of a Banach
space were studied in \cite{DG2,DEG} (cf \cite{GL1}). Using the techniques in question
we can regard
spaces such as $\Lambda$ (and likewise, $\Gr(p,A), \pol$, etc.)
as \emph{analytically} parametrized in terms of analytic maps
$D_0 \lra \Lambda$
to the underlying Banach space (of $\Lambda$), where $D_0 \subset \bC$
denotes the open unit disk. Thus taking $w \in D_0$ as a local parameter,
one can assign (operator) $\Lambda$-valued
functions $\zeta_{\lambda}(w)$ parametrizing $\lambda \in \Lambda$
(and likewise
for, e.g., $W \in \Gr(p,A)$). With this understood, we shall simply
write, as a convention,
$\zeta$ for $\zeta(w)$ and $z$ for $z = f(w)$, etc.,
in the following. We also take $\Hol(D_0,\Lambda)$ to denote the space
of holomorphic (analytic)
$\Lambda$-valued functions on $D_0$.

 Following \cite{Zelikin}, we  assume that
(commensurable) subspaces in $\Gr(p,A)$ are
isomorphic to those of $\Gr^*(p,A)$ and consider a smooth family
of subspaces $\cH(s) \in \Gr(p,A)$ parametrized by one real parameter $s$,
where $\cH(u) \cong \hp$ for $u > 0$, and $\cH(v) \cong \hm$ for
$v \leq 0$, so that
the pair $(\cH(u), \cH(v))$ defines a polarization of $H_{\A}$.
More specifically, consider
an ordering $s_2 < 0 <s_1 < s_3$, and a pair of polarizations
identified with points
$(\hp, \hm), (\kp, \km) \in \pol$:
\begin{equation}\label{pair-pol}
\begin{aligned}
(\cH_2, \cH_1) : &= (\cH(s_2), \cH(s_1)) \cong (\hp, \hm), \\
(\cH, \cH_3) : &= (\cH(0), \cH(s_3)) \cong (\kp, \km).
\end{aligned}
\end{equation}
For spaces such as $\Gr(p,A), \Lambda$ and $\pol$
(we recall that $\Lambda$ and $\pol$ are analytically diffeomorphic),
we define \emph{the operator cross-ratio} (`$\cross$') in terms
of projection ($A$-valued)
affine coordinates as
\begin{equation}\label{cross}
\cross(a,b;c,d) = (a - c)(a - b)^{-1}(b - d)(c -d)^{-1}.
\end{equation}
Let $z$ be such a $\Lambda$-valued variable in $\cH(s)$, and
letting $z_i$ denote variables with respect to
$\cH_i$ ($z$ corresponds to $\cH(0)$), we apply \eqref{cross} to the
polarizing pair in \eqref{pair-pol}
to obtain
\begin{equation}\label{cross-1}
\cross(\cH_2, \cH_1; \cH, \cH_3) = \cross(z_1,z_2;z,z_3)
= (z_1 - z)(z_1 - z_3)^{-1}(z_2 - z)(z_2 - z_3)^{-1},
\end{equation}
which is defined on projection ($A$-valued) analytic coordinates of $\pol$.
As shown in \cite{Zelikin}, this yields a
$\End(\gamma_{\Gr})$-valued $1$-cocycle $\{ \cross \}
\in H^1(\Gr(p,A), \End(\gamma_{\Gr}))$,
and hence under the pullback
\begin{equation}\label{pull-back-1}
\pi^*_{\Lambda}: H^1(\Gr(p,A), \End(\gamma_{\Gr}))
\lra H^1(\Lambda, \End(\gamma_{\Lambda})),
\end{equation}
we regard $\{\cross\}$ as also an $\End(\gamma_{\Lambda})$-valued
1-cocycle on $\Lambda$. Likewise the $\T$-function
can be viewed as a $G(pAp)$-valued $1$-cocycle in terms of the
transition function $t_{\Lambda}$ as
given by \eqref{T-FUNCTION-1}(see \cite[\S8]{DGP3} for details).
The operator cross-ratio is used by Zelikin
 \cite{Zelikin} to introduce an
operator analogous to the Schwarzian derivative  \cite[\S4]{Zelikin} .
 The key idea  is that, though operators do not
commute, one can take limits within the cross-ratio
along the real parameter $s$, one at a time
for $s_2$, $s_3$, and $s_1$ as in \eqref{pair-pol}, checking at every step an
asymptotic polarization consistency in the process.
We can apply Zelikin's result verbatim for  $z$-curves ($z\in\bC$)
in $\Lambda$, and write (with respect to the 1-dimensional parameter $s$):
\begin{equation}\label{schwarzian}
\Schw_{\Lambda}(z) = (z')^{-1} z''' - \frac{3}{2} ((z')^{-1} z'')^2.
\end{equation}

\begin{remark}\label{schw-deriv}
The Schwarzian derivative,
in the classical case of scalar-valued functions,
 arises naturally (cf. \cite{ZelikinBook},
where it is derived from the Wronskian for a basis
of solutions of a third-order differential equation obtained
by writing the invariance of the cross-ratio under
linear-fractional transformation, differentiating both sides, and
eliminating the three parameters of $\mathrm{PSL}(2)$)
and is in fact the only projectively
invariant 1-cocycle on $\mathrm{Diff}
(\mathbb{R}\mathbb{P}^1)$ \cite{OT1}.
The significance of Zelikin's definition
 \cite{Zelikin} rests partly on the fact
that, for matrix-valued deformations $z(s)$,
he was able  to show that the Schwarzian operator preserves
the operator cross-ratio
(for extensions of linear fractional transformations/cross-ratio to the
operator-valued
setting, along with applications, see e.g. \cite[Ch. 3]{Helton}).
\end{remark}

We now apply properties of this Schwarzian to our setting and proceed to define:
\begin{equation}\label{hol3}
\Hol^{(3)}(D_0,\Lambda) = \{ f \in \Hol(D_0,\Lambda):
\text{values of the derivatives $f^{(n)}$ commute for $n \leq 3$},
f'(z) \neq 0\}.
\end{equation}

Taking  $f \in \Hol^{(3)}(D_0,\Lambda)$ with $z=f(w)$ as before, we next
consider
\begin{equation}
(S_{\Lambda}f)(w,t) = \cross(f(w+ ta), f(w + tb); f(w + tc), f(w+ td)),
\end{equation}
Following \cite{Ahlfors} (cf. \cite{OT1}), we deduce formally
from the vector-valued case that, to second order,
\begin{equation}\label{ahlfors-1}
(S_{\Lambda}f)(w,t) = \cross(a,b;c,d)(1 + \frac{1}{6}(a -b)(c -d)
 (\Schw_{\Lambda}f)(w) t^2 + o(t^2)),
\end{equation}
confirming the Schwarzian $\Schw_{\Lambda}f$ to be
the infinitesimal version of the cross-ratio
that does not change in first order (cf. \cite{Zelikin}).

\begin{remark}\label{proj-remark}
In direct analogy with the classical case \cite{Gunning}
(see also \S\ref{proj-struct})
we may regard those $f \in \Hol^{(3)}(D_0,\Lambda)$ for
which $S_{\Lambda}f = 0$ as \emph{projective transformations}
and in particular, a coordinate system for which
$S_{\Lambda}(z) = 0$ could be taken to define a \emph{projective
structure on $\Lambda$}.
\end{remark}

Let us now recall the function $\T_{\lambda}$ from Definition \ref{def-T-2}.
In view of \eqref{tauop1}, and on setting $\Delta \T_{\lambda}(t)
= (1 + \frac{1}{6}(a -b)(c -d) (\Schw_{\Lambda}f)(w) t^2 + o(t^2))$,
we directly deduce from \eqref{ahlfors-1}
 on applying the $\T_{\lambda}$-function to \eqref{cross-1},
the relationship
\begin{equation}\label{ahlfors-2}
\T_{\lambda}(\cH_2, \cH_1;\cH, \cH_3) ~\Delta \T_{\lambda}(t)
 = S_{\Lambda} (\cH_2, \cH_1; \cH, \cH_3)(t),
\end{equation}
where we have implicitly used the pull-back of
the cross-ratio under $\pi_{\Lambda}$ in \eqref{pull-back-1}
above, along with \eqref{T-relation-1}.
In view of the formulas in \eqref{curv-2-form},
 we next consider the infinitesimal deformation of the
curvature $2$-form of $\nabla (\Det_{\Lambda})$ under an element of
$\Hol^{(3)}(D_0,\Lambda)$:

\begin{proposition}\label{2-forms}
Let $f \in \Hol^{(3)}(D_0,\Lambda)$. Then we have the following `
asymptotic' relationship under $f$ between $2$-forms in
\eqref{curv-2-form}:
\begin{equation}\label{omega-1}
\Omega(\Det_{\Lambda})  - \Omega(\Det_{f(\Lambda)})
= \frac{1}{2\pi \iota} \del_{+}
\del_{-} \log \vert \det (\Delta \T_{\lambda}(t)^{-1}) \vert.
\end{equation}
\end{proposition}

\begin{proof}
 From \eqref{curv-2-form}, we have on applying `$\det$' to
\eqref{ahlfors-2}
the following:
\begin{equation}\label{ahlfors-3}
\begin{aligned}
\Omega(\Det_{\Lambda}) &= \frac{1}{2\pi \iota} \del_{+}
\del_{-} \log \vert \tau_{\lambda} \vert \\
&= \frac{1}{2\pi \iota} \del_{+}
\del_{-} \big[ \log \vert \det (\Delta \T_{\lambda}(t)^{-1}) \vert + \
\log \vert \det ((S_{\Lambda})f(w,t)) \vert \big].
\end{aligned}
\end{equation}
But by definition, $(S_{\Lambda})f(w,t) = f \circ \T_{\lambda}$, and so
 we have
\begin{equation}
\det((S_{\Lambda}f)(w,t)) = \det (f \circ \T_{\lambda}) = \tau_{f(\Lambda)}.
\end{equation}
Thus the right-hand side of \eqref{ahlfors-3} becomes
\begin{equation}
\frac{1}{2\pi \iota} \del_{+}
\del_{-} \big[ \log \vert \det (\Delta \T_{\lambda}(t)^{-1})
\vert + \log \vert  \tau_{f(\Lambda)} \vert  \big],
\end{equation}
from which the result follows in view of \eqref{curv-2-form} (observe
that essentially the same
applies to $\Omega(\Det_{\Gr})$ in \eqref{curv-2-form}).
\end{proof}

\begin{remark}
In view of his definition  of the operator Schwarzian derivative
 on one-dimensional submanifolds of the Grassmannian,
 Zelikin posits, in a speculative manner, that
the operator-valued KP deformations (whose Baker functions
satisfy Riccati's equation,
cf. also \cite{Segal}) might  be studied using the
Schwarzian (of
which he proved that
a quotient of solutions, which defines a projective
structure, also satisfies a Riccati equation); but, quote \cite[p. 51]{Zelikin}: ``Unfortunately,
the trajectories of [the Riccati] fields do not lie in
[the restricted Grassmannian], not even on arbitrarily small time intervals''.
Here we take a different route.  Raina \cite{Raina},
motivated by conformal-field theory, re-wrote Fay's trisecant
identity (the Riemann-surface version of the KP equation)
as a generalized cross-ratio; below (\S\ref{proj-struct}-\ref{k-dv})
we use that strategy, together with the operator
cross-ratio for the Riemann surface, to arrive at the operator KP equations.
\end{remark}



\subsection{Relationship with projective structures on $X$}\label{proj-struct}

We turn to the work in \cite{BR, Raina};
we apply Zelikin's Schwarzian operator (\S\ref{cross-ratio}),
extended to our setting, to the
description of the projective structure on a
compact Riemann surface $X$ given in \cite{BR}.
 We are ultimately  able (\S\ref{wick}) to  re-derive the KP hierarchy
with operator coefficients that was the motivation behind \cite{DGP2}.

Let $K_X$ denote the canonical line bundle of $X$. The space of all
projective structures on $X$
is an affine space for the complex
 vector space $H^0(X, K_X^2)$
of global holomorphic quadratic differentials, in which the classical
Schwarzian $\Schw_{X}$
determines a cocycle (see \cite[pp. 170--172]{Gunning}).

More explicitly, take a covering
$\{U_i\}_{i \in \mathcal I}$ for $X$ and let $\phi_i$ be a
holomorphic function on $\zeta_i\in U_i$, so that the transition
function $h_i$ for $H^0(X, K_X^2)$, satisfies (see e.g. \cite{BR,Gunning})
\begin{equation}\label{schw-1}
h_i = \Schw_{X}(\phi_i) \zeta_i.
\end{equation}
Observe that if $\phi'_i$ is another function
satisfying \eqref{schw-1}, then $\phi'_i(z_i) = \Xi \circ \phi_i(z_i)$
where $\Xi$ denotes a M\"obius transformation.

Referring back to \S\ref{cross-ratio},
as noted in \cite{Zelikin} (in the matrix-valued case),
the definition of cross-ratio and Schwarzian can be extended
to complex deformations in one parameter $s$.
In such a situation we denote the Schwarz operator by $\Schw_{\Lambda}(s)$.
We focus on the deformations taking place along the curve (the compact Riemann
surface) $X$ embedded in the $\Gr(p,A)$ by the Krichever map:
recall
\cite[Appendix \S A.4]{DGP2} that to define the Krichever map we need to
fix a local parameter at a point $\infty$, a
(generic) line bundle and a local trivialization; with these data,
we produce a point $W\in \Gr(p,A)$; the action of $\Gamma_1 = \Gamma_+(\A)$
on $W$ sweeps out $J_{\A}(X)$, so by the Abel map
we have (up to several choices) a holomorphic embedding
$X \hookrightarrow\Lambda$ (see Proposition \ref{proj-prop} below).
In analogy with the classical case \cite{Gunning}
(see also Remark \ref{proj-remark})
we  regard those $z$ for which $\Schw_{\Lambda}(z) = 0$ as projective
coordinates, thereby defining
a projective
structure on the embedded copy of $X$.

The following proposition summarizes certain analytic properties of $\Lambda$
(which like $\Gr(p,A)$ is modeled on a complex Banach space):
\begin{proposition}\label{lambda-analytic}
The space $\Lambda$ is an open and closed holomorphic (Banach)
submanifold of $P(A)$
which is a holomorphic (Banach) submanifold of $A$.
\end{proposition}
\begin{proof}
That $\Lambda$ is an open and closed holomorphic submanifold of $P(A)$
has been shown to be the case in \cite{DG2,DGP3}(cf. \cite{Raeburn}).
In fact, $\Lambda$ is  locally a holomorphic retract of $A$, as seen as
follows. For $x,y \in A$, we define $g(x,y) = xy + (1-x)(1-y)$,
noting that $g(p,p) =1$ and therefore invertible for all $x,y$
in some open subset $U$ of $A$ containing $p$.
We note then for $q \in \Lambda$, that $pg(p,q) = g(p,q)q$,
so for $q \in \Lambda \cap U$, we have $g(p,q)^{-1}pg(p,q) = q$.
Thus $r(x)= g(p,x)^{-1}pg(p,x)$ is a holomorphic retraction of $U$ onto
its overlap with $\Lambda$, on shrinking $U$ further if necessary.
This also shows that $\Lambda$ is a holomorphic submanifold of $A$.
\end{proof}
We summarize the above facts relating to holomorphic/projective structures
in the following proposition which
compares, via $\Schw_{\Lambda}$,
the complex structure induced from $\Lambda$
with the one intrinsic to $X$.


\begin{proposition}\label{proj-prop}
Let $\eta: X \hookrightarrow \Lambda$ be a holomorphic
embedding with respect to the
natural (complex) analytic structure of $\Lambda$. Then
the Schwarzian operator intrinsic to $X$ and the one induced from $\Lambda$,
with respect to the holomorphic
deformation along the embedded curve, correspond.
\end{proposition}

Returning to the cross-ratio class $\{\cross\}
\in H^1(\Lambda, \End(\gamma_{\Lambda}))$ in \S\ref{cross-ratio}
leading to the Schwarzian $\Schw_{\Lambda}$, we see that for functions
$f,h \in \Hol^{(3)}(D_0,\Lambda)$, it is formally
deduced from the classical case that
\begin{equation}
\Schw_{\Lambda}(h \circ f) = (\Schw_{\Lambda}(h)
\circ f)(f')^2 + \Schw_{\Lambda}(f),
\end{equation}
(see e.g. \cite{Ahlfors,Osgood,OT1}). Here the first right-hand summand is the
action of $f$ on a quadratic differential
\begin{equation}
(u \circ f)(z) = u(f(z) \cdot f'(z))^2,
\end{equation}
in terms of the $z$-coordinate above, in turn leading to a transition
 function for the vector bundle
$K^2_{\Lambda}$ in an analogous way to the classical situation.
 We thus have arrived at the novel concept of an \emph{operator-valued
projective structure} induced on a
one-dimensional complex submanifold from a space such as $\Lambda$.

\begin{remark}
It would be interesting to apply the
operator Schwarzian derivative
to compare the deformations
of $X$ inside $\Lambda$ with the deformations of $X$ parametrized by
$H^0(X,K^2_X)$, especially since the former should be
unobstructed (Krichever's map can be applied to any Riemann surface,
and locally it should be possible to make consistent
choices of a defining quintuple, cf. \cite{SW}).
However, the tangent bundle to $\Lambda$,
whose first cohomology gives the deformations, is
of infinite rank, and in order to define a canonical line bundle
over $\Lambda$ requires different techniques (cf \cite{BG}), as
 $\Lambda$ is an infinite-dimensional
Banach manifold which is a holomorphic submanifold of $A$, thus in
general, since we are in infinite dimensions, the top
exterior power is not formed in the usual way. However, in our
specific situation, where  $A$ is the restricted algebra, we have a group
transforming the restricted frames in $V(p,A)$
(``admissible bases'' in \cite[\S3]{SW})
on which the determinant is defined, and a central extension
of it ($\mathcal{E}$ in \cite[\S3]{SW}) where
the function $g$  giving the retraction
in Proposition \ref{lambda-analytic} takes values, so
that coordinate transformations in effect have derivatives which have
values in that central extension
($\mathcal{E}$) as well.
With this in hand,  we would need to restrict the
line bundle thus obtained to $X$
and see if each projective structure of $X$ can
be extended away from it, not only along the one-dimensional
deformations controlled by the Schwarzian derivative.
\end{remark}

\subsection{The Riemann theta function and Wick's theorem}\label{wick}

Our next main observation concerns (generalized) projective structures on
$X$  using a `correlation function'
approach motivated by {Wick's theorem} and a (generalized) cross-ratio
(see \cite{BR,Raina}). The point of this subsection is to
implement Raina's rendition of Fay's trisecant identity in terms
of the cross-ratio we developed
(\S\ref{cross-ratio}), and obtain the KP hierarchy in
\S\ref{wick} below, in our extension-group model.
Here we utilize
the Burchnall-Chaundy C*-algebra $\uA$ and the extension
group $\Ext(\uA)$ from \cite[\S4]{DGP2}.

\begin{theorem}\label{main-2}

For genus $g_X \geq 2$,
the action of the group $\Gamma_1 =\Gamma_{+}(\A)$ on $\Ext(\uA)$
corresponds to translating the theta function of $X$ on the Jacobian.
\end{theorem}

\begin{proof}
Let $\mathfrak{L}_{\theta} \lra J_{\A}(X)$ be the holomorphic
 line bundle whose sections are theta
functions $\theta[\xi](z)$ of characteristic $\xi$ (see e.g. \cite{Fay,GH})
taken as $\bC \otimes 1_{\A}$-valued.
Similar to before, consider a $\Lambda$-parametrization of the
surjective homomorphism of
\cite[A(19)]{DGP2}, giving a commutative diagram
\begin{equation}
\begin{CD}
\Gamma_1 \times \Lambda @> \tau_{\lambda}
>> \bC \otimes 1_{\A}
\\ @VVV   @VVV\\
J_{\A}(X) \times \Lambda @>
\theta_{\lambda}[\xi]
>> \bC \otimes 1_{\A}
\end{CD}
\end{equation}
The commutativity of this diagram reveals
that $\tau_{\lambda}$ is `proportional' to
$\theta_{\lambda}[\xi]$ following the classical case
 (cf. \cite{Raina,SW}), and hence from
\cite[Theorem 4.5]{DGP2}, $\Ext(\uA)$ parametrizes a family
of (translations of the)
theta function(s) via extensions by the compact operators.

Next, consider an (operator)-valued spinor
field $\Vec{\psi}$ on $X$,
and points $a,b \in X$ suitably chosen to lie in the same coordinate patch.
In \cite{Raina} it is shown that the Fay trisecant identity
(see \eqref{tris-1}) is equivalent to \emph{Wick's theorem},
 which in terms of a left-hand-side `correlation function' below
(see \cite{BR,Raina}),
 can be expressed in the form
\begin{equation}\label{tris2}
\langle \Vec{\psi}^*(b_1) \Vec{\psi}(a_1) \Vec{\psi}^*(b_2)
\Vec{\psi}(a_2) \rangle =
\det \bmatrix S_{\a}(b_1, a_1) & S_{\a}(b_1, a_2) \\ S_{\a}(b_2, a_1) &
S_{\a}(b_2, a_2)
\endbmatrix
\end{equation}
where $S_{\a}$ is the Szeg\"{o} kernel of a theta prime-form \cite{Fay}.
Moreover, in view of \cite[\S5]{BR} the `correlation function' of
\eqref{tris2}
defines a projective connection and hence a projective structure
\cite{Gunning}).
\end{proof}

\begin{remark}
In terms of the Cauchy kernels of flat vector bundles of arbitrary
rank on $X$ (for $g_X \geq 1$), explicit
formulas in \cite{Ball3} interpolate homomorphisms of these bundles
 leading to generalizations of the Fay
trisecant identity. We note that flat rank-2 bundles over $X$
(for $g_X \geq 2$) have been related \cite{BR}
to quadratic differentials (in both cases
the moduli spaces have the same dimension, $3g-3$).
It is possible this may lead to further
applications of operator
 projective structures on $X$
as well.
\end{remark}

\subsection{The KP tau-function and trisecant identity}\label{k-dv}

We briefly recall the concept of the KP-hierarchy starting
 from \cite[Appendix \S A4]{DGP2}.
Consider a formal pseudodifferential operator of the form
\begin{equation}
L= \del + a_0\del^{-1} + a_1\del^{-2} + \cdots,
\end{equation}
where as in \cite[\S A4]{DGP2} we take the
coefficients $a_i = a_i(t_1, t_2, \ldots)$
 to be $\A$-valued functions.
Next, let us set $P^{(k)} = (L^k)_{+}$. Then for each $k \in \mathbb{N}$,
 the \emph{KP hierarchy} (see e.g. \cite{Krich1,Sato,Segal,SW}) consists
of partial differential equations of the type
\begin{equation}
\frac{\del L}{\del t_k} = [P^{(k)}, L],
\end{equation}
according to which
\begin{equation}
\frac{\del}{\del t_k}P^{(\ell)} - \frac{\del}{\del t_{\ell}}P^{(k)} +
 [P^{(\ell)}, P^{(k)}] = 0.
\end{equation}
Since we are using $\A$-valued coefficients, we choose to denote this
hierarchy by $\KP(\A)$. Once more we apply the extension
group $\Ext(\uA)$ and establish the following for the
$\KP(\A)$ flows with respect to the $\Gamma_1$-action:

\begin{theorem}\label{main-1}
The $\KP(\A)$ flows evolve on
the group $\Ext(\uA)$ via extensions by compact operators.
\end{theorem}

\begin{proof}
Recalling the $z$-coordinate of \S\ref{cross-ratio}, we set
 $[z] = (z, z^2/2, z^3/3, \ldots)$.
Then from the operator-valued  $\tau$-function $\tau_{\lambda}:
\Gamma_1 \lra \bC \otimes 1_{\A}$
in \S\ref{poincare-1}
(cf. \cite{DGP2}), it follows that
(cf. \cite{SW,Sato})
\begin{equation}\label{kdv-tau}
\tau_{\lambda}(t + [z]) = \tau_{\lambda}(t_1 + z, t_2 + z^2/2, t_3 + z^3/3,
\ldots),
\end{equation}
that satisfies $\tau_{\lambda}(t + [-z]) = \tau_{\lambda}(t - [z])$.
Further, on setting $C(z_i, z_j, z_k, z_{\ell}) =
(z_i - z_j)(z_k - z_{\ell})$ in $\A$-valued variables, for
$0 \leq i,j,k, \ell \leq 3$, we formally deduce from references
\cite{Fay,Miwa} the
\emph{Fay trisecant identity}:
\begin{equation}\label{tris-1}
\begin{matrix}
C(z_0, z_1, z_2, z_3) \tau_{\lambda}(t + [z_0] +[z_1]) \tau_{\lambda}
 (t + [z_2] +[z_3]) +\\
C(z_0, z_2, z_3, z_1) \tau_{\lambda} (t + [z_0] +[z_2]) \tau_{\lambda}
 (t + [z_3] +[z_1]) + \\
C(z_0, z_3, z_1, z_2) \tau_{\lambda}(t + [z_0] +[z_3]) \tau_{\lambda}
(t + [z_1] +[z_2]) = &0.
\end{matrix}
\end{equation}
from the analogous identity for $\tau_W$ and then using \eqref{tau-lambda}.

 From \cite[Theorem 4.5]{DGP2}, elements of $\Ext(\uA)$, extensions
by the Burchnall-Chaundy C*-algebra $\uA$ of the compact operators,
lead to the map
\begin{equation}
\Upsilon^{-1}_{\A}: \Ext(\uA) \lra \Gamma_1
\end{equation}
(the inverse of the map $\Upsilon_{\A}$ as in the proof of \cite[Theorem 4.5,
 (4.20)]{DGP2}),
which yields a family of $\tau$-functions, one for each such extension.
Hence, in each case a corresponding trisecant identity as in \eqref{tris-1}
follows. But \eqref{tris-1}
has been shown to be equivalent to the KP-hierarchy (see e.g. \cite{Miwa}).
Hence we conclude that
the KP-hierarchy, implemented by the $\Gamma_1$-action,
flows on these extensions by compact operators as derived from $\Ext(\uA)$.
\end{proof}




\end{document}